\documentclass[]{amsart}

\usepackage{amscd,amsthm,amsfonts,amssymb,amsmath,esint}
\usepackage{amsmath,tikz-cd}
\usepackage[colorlinks=true]{hyperref}
\usepackage{fullpage}
\usepackage[all]{xy}
\usepackage{mathtools}
    \newcommand{\BC}{{\mathbb {C}}} 
     \newcommand{\BF}{{\mathbb {F}}}
     \newcommand{\BH}{{\mathbb {H}}}

    \newcommand{\BQ}{{\mathbb {Q}}} \newcommand{\BR}{{\mathbb {R}}}

     \newcommand{\BZ}{{\mathbb {Z}}}

     \newcommand{\CF}{{\mathcal {F}}}

     \newcommand{\CL}{{\mathcal {L}}}
     
    \newcommand{\CO}{{\mathcal {O}}}

     \newcommand{\CX}{{\mathcal {X}}}

    \newcommand{\Ar}{{\mathrm{Ar}}}

    \newcommand{\coker}{{\mathrm{coker}}}

    \renewcommand{\Im}{{\mathrm{Im}}}

     \newcommand{\rank}{{\mathrm{rank}}}
     \newcommand{\Pic}{\mathrm{Pic}}
    
    \newcommand{\PSL}{{\mathrm{PSL}}}

    \newcommand{\covol}{{\mathrm{covol}}}
    \newcommand{\Spec}{{\mathrm{Spec}}}

    \newcommand{\vol}{{\mathrm{vol}}}

    \newcommand{\wh}{\widehat}
    
    \newcommand{\pair}[1]{\langle {#1} \rangle}

    \newcommand{\ov}{\overline}

    \newcommand{\sk}{\medskip}
    
    \newcommand{\ra}{\rightarrow}

    \newcommand{\s}{\sk\noindent}

    \theoremstyle{plain}
    \newtheorem{thm}{Theorem}[section] 
    \newtheorem{lem}[thm]{Lemma}  \newtheorem{prop}[thm]{Proposition}\newtheorem{def-prop}[thm]{Definition and Proposition}
     \newtheorem{defn}[thm]{Definition}

\theoremstyle{remark} \newtheorem{remark}[thm]{Remark}
\theoremstyle{remark} 
\theoremstyle{remark} 

    \numberwithin{equation}{section}

\begin{document}

\title{Numerical cohomology for arithmetic surfaces and applications}
\author{Wei He}
\address{School of Mathematics and Statistics, Xi'an Jiaotong University, Xi'an 710049, P.R. China.} 
\email{hewei0714@xjtu.edu.cn}
\begin{abstract}In this paper, we introduce numerical cohomology for arithmetic surfaces, which leads to an absolute version of arithmetic Riemann-Roch formula. As an application, we derive an upper bound for the self-intersection number of relative dualizing sheaf in terms of successive minima with respect to $L^2$-norm. The result has the geometric analogue that the slopes of the Harder-Narasimhan filtration of relative dualizing sheaf provide an upper bound for self-intersection number. Suppose that the arithmetic surface admits a section and has generic fiber of genus at least two, we obtain a refined upper bound for the self-intersection number, which is governed by the topological and arithmetic information of the section.\end{abstract}
\maketitle

\tableofcontents

\section{Introduction}\label{sec1}
Let $f: \CX\ra \Spec\CO$ be an arithmetic scheme, where $\CO$ is the ring of integers of a number field $F$. Consider a Hermitian coherent sheaf ${\CF}$ on $\CX$.
One arithmetic analogue of $h^0$ was defined by logarithm of the cardinality of small sections
\[h^0(\CF)=\log\#\{x\in H^0(\CX,\CF)\ |\ ||x||_{v}\leq 1,\quad \forall v|\infty\},\]which motivated by geometry of numbers. For $\CX=\Spec\CO$, this definition essentially appeared in Weil's note \cite{Weil:1939}. 

Such definition may not be so satisfactory since it does not behave like traditional geometric analogues. There is another arithmetic analogue of $h^0$ for arithmetic curves, which was defined by the logarithm of theta invariant\footnote{it was also called Gauss mass}
\[h_{\CO}^0(\CF)=\log \sum_{x\in H^0(\CX,\CF)}e^{-\pi ||x||^2}.\]
The origin of the definition of $h_\CO^0$ appeared as one side of Poisson summation formula, may be dated back to Artin, Hasse, Tate \cite[Sec.~4.2]{Tate:67}, and Quillen \cite{Quillen}. See \cite[Sec.~0.1.2]{Bost:2020} for more details and \cite{Groenewegen}, \cite{Geer-Schoof}, \cite{Bost:2020} for recent developments.

Let $\chi_\CO(\CF)=h^0_\CO(\CF)-h^0_\CO(\omega_{\CO}\otimes\CF^\vee)$ be the Euler characteristic of $\CF$.
As in the characteristic $p$ geometric case, the Riemann-Roch formula 
\[\chi_\CO(\CF)=\wh{\deg}\det\CF+\chi_\CO(\CO)\rank_{\CO}\CF,\quad \] is equivalent to Poisson summation formula. Furthermore $h_{\CO}^0(\CF)$ is closely related to $h^0(\CF)$ but has better properties \cite[Preface~\&~Sec.~3.1]{Bost:2020}. 

Now consider $\CX$ to be an arithmetic surface and $\CL$ to be a Hermitian line bundle on $\CX$. Arakelov, in his original paper \cite{Arakelov:1974}, introduced a version of $h^0$ that uses a `norm' on $H^0(\CX, L)$, which is closely related to archimedean intersection. Then he made a conjectural arithmetic Riemann-Roch formula and a formula for the residue of height zeta function. 

Later, the arithmetic Riemann-Roch was proved by Faltings \cite{Faltings} with arithmetic analogue of Euler characteristic defined by arithmetic degree on $\det Rf_*\CL$ with Faltings metric. There are also arithmetic Riemann-Roch for arithmetic surfaces proved by Deligne \cite{Deligne:85} with metric on $\det Rf_*\CL$ to be Quillen metric, which has a generalization to higher dimensional arithmetic schemes by Gillet-Soul\'{e} \cite{Gillet-Soule}.

For arithmetic surfaces, it is natural to seek a good notation for the numerical cohomology $h^i$ ($i=0,1,2$), analogous to the one used in the geometric setting. Also see \cite[Conj.~2.1.1]{Zhou2} on a conjecture of Bost on positivity of $H^1$ and \cite[Sec.~9]{Geer-Schoof} for a guess of van der Geer and Schoof for $h^0$ and $h^2$ which is close to Arakelov's one. A similar conjecture for numerical cohomology of arithmetic surfaces was also proposed in \cite{Borisov}.

In this paper, we introduce a numerical cohomology \[h^i_{\CX}(\CL),\quad  i=1,2,3\] for a Hermitian line bundle $\CL$ on $\CX$. The motivation is the Leray spectral sequence and Serre duality. Same as the geometric case, which allows us to define numerical cohomology of Hermitian line bundles on $\CX$ via numerical cohomology $h_\CO^0$ of Hermitian coherent $\CO$-modules. As a result, we obtain the following absolute arithmetic Riemann-Roch for numerical cohomolgy, which deduced from Arithmetic Riemann-Roch of Faltings, Deligne and Gillet-Soul\'{e}.

In the following, let $\CL$ be a Hermitian line bundle on $\CX$. Fix the Arakelov metric on relative dualizing sheaf $\omega_{\CX/\CO}$. There is a natural Hermitian structure on $H^0(\CX,\CL)$ given by the induced $L^2$ metric. Let $\omega_{\CX}=\omega_{\CX/\BZ}$, which we may call canonical sheaf.
Motivated by the arithmetic curve case and Serre duality for algebraic surfaces, define
\[h^0_{\CX}(\CL):=h^0_{\CO}(f_*\CL).\] and \[h^2_{\CX}(\CL):=h^0_{\CX}(\omega_{\CX}\otimes \CL^\vee).\]
Inspired by Leray spectral sequence: \[E_2^{p,q}:=H^p(\CO,R^qf_* \CL)\Rightarrow H^{p+q}(\CX,\CL),\] and relative Serre duality, define 
 \[\begin{aligned}h_\CX^1(\CL):=&h_\CO^0(\omega_\CO\otimes (f_*\CL)^\vee)+h_\CO^0(f_*(\omega_{\CX/\CO}\otimes\CL^\vee)^\vee)\\
  +&\frac{1}{2}(\wh{\deg}\det H^1(\CX,\CL)_t+\wh{\deg}\det H^1(\CX,\omega_\CX\otimes\CL^{\vee})_t),\\
\end{aligned}\]where $H^1(\CX,\CL)_t\subset H^1(\CX,\CL)$ is the torsion $\CO$-submodule. Then $h_\CX^1(\CL)=h^0_{\CX}(\omega_{\CX}\otimes\CL^\vee)$.

Now define the Euler characteristic by \[\chi_{\CX}(\CL)=h^0_{\CX}(\CL)-h^1_{\CX}(\CL)+h^2_{\CX}(\CL),\] then it satisfies the following Riemann-Roch:
For each infinite place $v$, let $\det \Delta_{\CL_v}=e^{-\zeta_{\CL_v}(0)}$ be the regularized determinant of Laplacian (analytic torsion) of Hermitian line bundle $\CL_v$.  Let $\det \Delta_{\CL_\infty}=\prod_{v|\infty}\det \Delta_{\CL_v}$.

\begin{thm}\label{main1}Let $\CL$ be a Hermitian line bundle on an arithmetic surface $\CX$ and let $\omega_{\CX}$ be the canonical sheaf equipped with the Arakelov metric, then
\[\left(\chi_\CX(\CL)+\frac{1}{2}\log \det \Delta_{\CL_\infty}\right)=\frac{1}{2}(\CL,\CL\otimes \omega_{\CX}^\vee)+\left(\chi_\CX(\CO_{\CX})+\frac{1}{2}\log \det\Delta_{\CO_{\CX_\infty}}\right),\]where $(\ ,\ )$ is the Arakelov intersection pairing on $\wh{\Pic}(\CX)$.
\end{thm} 

\begin{remark}
In the geometric case, there exists a purely harmonic analysis proof of Riemann-Roch theorem for algebraic surfaces defined over a finite field by Osipov-Parshin \cite{Osipov-Parshin:2011I}. It is interesting to ask a purely harmonic analysis proof of Theorem \ref{main1}, see \cite{Parshin:2011}, \cite{Osipov-Parshin:2011II} for related development.
\end{remark}

The remaining part of the introduction is devoted to an application of the numerical cohomology for arithmetic surfaces defined as above.

 Our second result is an upper bound of \[(\omega_{\CX/\CO},\omega_{\CX/\CO})\] in terms the successive minima of $f_*\omega_{\CX/\CO}$, the genus $g$ of $\CX_F$ and arithmetic invariants of $F$.
\begin{thm}\label{main:int2}
Let $\lambda_i$, $i=1,\cdots, g$ be the successive minima of $f_*\omega_{\CX/\CO}$ with repsect to $L^2$-norm, then there exists an explicit constant $C=O(g\cdot([F:\BQ]\log g +|D_F|))$ such that 
\[(\omega_{\CX/\CO},\omega_{\CX/\CO})\leq 12[F:\BQ]\sum_{i=1}^{g}(-\log\lambda_i)+C.\]
\end{thm}
The result has the following geometric analogue: Let $f: S\ra C$ be a fibration of curves, where $S$ is a smooth projective surface and $C$ is a smooth projective curve.  Then 
\[(\omega_{S/C},\omega_{S/C})\leq 12 \deg f_*\omega_{S/C}\leq 12\sum_{i=1}^k r_i\cdot \mu_i,\]where $r_i$ (resp. $\mu_i$) is the rank (resp. slope) of the $i$-th graded piece of $f_*\omega_{S/C}$ with respect to the Harder-Narasimhan filtration. See Remark~\ref{ns} for the analogue between slope of Harder-Narasimhan filtration and successive minima for Hermitian vector bundles over $\Spec\CO$.

We now briefly explain the approach for Theorem \ref{main:int2}. As in the geometric case, $\wh{\deg}f_*\omega_{\CX/\CO}$ is closely related to $\chi_{\CX}(\CO_{\CX})$, since $\chi_{\CX}(\CO_{\CX})$ is essentially $\chi_{\CO}(f_*\CO)+\chi_{\CO}(f_*\omega_{\CX/\CO})$. Based on a Noether formula for $\chi_{\CX}(\CO_{\CX})$ (cf.~Proposition~\ref{Noe}), which due to Faltings and Moret-Bailly, it is reduced to bound the sum of $\chi_{\CX}(\CO_{\CX})$ and analytic torsion $\det \Delta_{\CO_{\CX_\infty}}$ from above. 
The upper bound for $\chi_{\CO}(f_*\omega_{\CX/\CO})$ via success minimum follows from Minkowski type theorem in geometry of numbers. In the arithmetic case, $\chi_{\CO}(f_*\CO_{\CX})$ is related to volume of $\CX_\infty$ with respect to Arakelov metric. This volume, together with analytic torsion, gives Faltings delta-invariant, where a lower bound was deduced by Wilms \cite{Wilms:2017}. Thus the result follows.

 \begin{remark}\
\begin{itemize}
  \item [(i)] Similar result for the upper bound of $(\omega_{\CX/\CO},\omega_{\CX/\CO})$ by any single successive minima of $f_*\omega_{\CX/\CO}^k$, $k\geq 2$ was obtained by Soul\'{e} \cite{Soule:94}. His proof is based on the arithmetic stability of rank two Hermitian vector bundles, which is the arithmetic counterpart of Mumford's proof of Bogomolov's theorem in the geometric case.
  \item [(ii)] We remark that an upper bound for $(\omega_{\CX/\CO},\omega_{\CX/\CO})$ in terms of minimum height with respect to $\omega_{\CX/\CO}$ was obtained in \cite[Thm.~5~b)]{Faltings}. Also see \cite[Thm.5~b)]{Elkik:1985} for a similar result based on Szpiro's geometric methods \cite{Szpiro:1981}. Also see \cite[Thm.6.3]{Zhang:1992} for $(\omega_{\CX/\CO},\omega_{\CX/\CO})$ in terms of essential minimum.
\end{itemize}

\end{remark}
 
Suppose that the following holds: \begin{itemize}
  \item [(A)] $\CX_F$ has genus $g>1$ and $\CX$ admits a $\CO$-section $P$.
\end{itemize}Then we have a lower bound for $\lambda:=\lambda_1$. Based on this, we obtain an upper bound for $(\omega_{\CX/\CO},\omega_{\CX/\CO})$ in terms of more refined data with respect to uniformization of hyperbolic curves $\CX_\infty$. The bound depends on $g$, $|D_F|$, $[F:\BQ]$, and on the topological and arithmetic properties of $P_\infty$. We now give the precise statement: 

Let $\BH$ be the upper half plane and $(\pi_v)_{v|\infty}:(\BH)_{v|\infty}\ra (\CX_v(\BC))_{v|\infty}$ be the universal covering with fundamental group $(\Gamma_v)_{v|\infty}\subset (\PSL_2(\BR))_{v|\infty}$. Choose any representatives $(z_v)_{v|\infty}\in (\BH)_{v|\infty}$ of $P_\infty$.
 
 Let $m_P\in \BR_{>0}$ be the multiupler of the local coordinate of $P_\infty$ around $(z_v)_{v|\infty}\in (\BH)_{v|\infty}$ (cf.~Definition~\ref{mul}). Let $r_{z_v}\in \BR_{>0}$ be the radius of maximal Euclide ball centered in $z_v$ that contained in the Dirichlet fundamental domain of $\CX_v(\BC)=\BH/\Gamma_v$ relevant to $z_v$.

  Let 
  \[c_P:=-\left(\log\min\{1,m_P^{2g-1}\}+\sum_{v|\infty}\frac{\epsilon_v}{2}\log \frac{2\min\{1,r_{z_v}^{4g-3}\}}{4g-3}\right).\]
 which only depends on $P$, $(z_v)_{v|\infty}$ and $g$, then 
 \[\log \lambda\geq -c_P\] holds (cf.~Proposition~\ref{mainrpt}).
Combine with Theorem~\ref{main:int2} we have the following:

Let \[C_P=12g[F:\BQ]\cdot \begin{cases}
 c_P,&\quad \text{if $\lambda< 1$},\\
 0,&\quad \text{otherwise}.
\end{cases}
\]
\begin{thm}\label{mainthm}There exists an explicit constant 
$C=O(g\cdot([F:\BQ]\log g +|D_F|))
$ such that 
\[(\omega_{\CX/\CO},\omega_{\CX/\CO})\leq C+C_{P}\]holds for $\CX$ satisfying assumption (A). 
\end{thm}

The approach is a generalization of the classical lower bound of Petersson inner product of normalized eigenforms. 
\begin{remark}\
  \begin{itemize}
    \item [(i)]For $\CX$ to be the integral model of modular curve or Fermart curves, good upper bounds of the $(\omega_{\CX/\CO},\omega_{\CX/\CO})$ are already known (cf.~\cite{Abbes-Ullmo, Michel-Ullmo, Kuhn, Curilla-Kuhn}). 
    \item [(ii)]The self-intersection number $(\omega_{\CX/\CO},\omega_{\CX/\CO})$ have rich arithmetic information \cite{Gasbarri}. A good effective upper bound or lower bound of $(\omega_{\CX/\CO},\omega_{\CX/\CO})$ would have many consequences. For example, a good lower bound implies Bogomolov conjecture \cite{Zhang:1993} and a good upper bound would imply the effective Mordell conjecture \cite{Parshin:1988, Soule, Vojta}.
  \end{itemize}

\end{remark}

\section*{Notations}

Let $F$ be a number field, denote by $\CO$ its ring of integers, and $|D_F|$ its absolute discriminant.
For $v$ a place of $F$, let
\[\epsilon_v=\begin{cases}
  \log q_v,&\quad  v\ \text{is finite with residue field $\BF_{q_v}$},\\
  1,&\quad \text{$v$ is real},\\
  2,&\quad \text{$v$ is complex}.\\
\end{cases}\]
Let $\wh{\deg}:\wh{\Pic}(\CO)\ra \BR$ be the arithmetic degree.
For an arithmetic surface $f:\CX\ra \Spec\CO$, let $g$ be the genus of $\CX_F$.
For a coherent $\CO$-module, we will use subindex $t$ (resp. $f$) for its torsion part (resp. free part). 

We will use $O(x)$ of a number $x\in \BR_{\geq 0}$ to mean bounded by the multiple of $x$ by an explicit positive constant that does not depend on $x$.
\section*{Acknowledgments}

The author expresses sincere gratitude to Ye Tian for constant encouragement and to Ping Xi for insightful discussions. The author is particularly indebted to C. Gasbarri, both for inspiring talks during the 2019 MCM workshop in Beijing and for subsequent enlightening discussions that greatly shaped the perspective of this work. The author is also grateful to Binggang Qu and Xiaozong Wang for helpful comments and discussions.
\section{An absolute arithmetic Riemann-Roch formula} 
This section introduces numerical cohomology for arithmetic surfaces. The main result is the arithmetic Riemann-Roch formula for numerical cohomology (cf.~Theorem~\ref{main:RR}).
\subsection{Preliminary}\label{pre}In Sections \ref{ssd} and \ref{Sqmrr}, we recall some basics of $L^2$-norm, Quillen metric, and Serre duality for Riemann surfaces, following Deligne \cite[Sec.~1.2~$\&$~1.4]{Deligne:85}. We also recall the relative version of arithmetic Riemann-Roch in Section \ref{arr}.
\subsubsection{Serre duality for Riemann surfaces}
\label{ssd}
Let $X$ be a compact Riemann surface equipped with a Hermitian metric $g$. Let $\Omega_X$ be the canonical sheaf of $X$ and let $L$ be a Hermitian line bundle on $X$.

Let $\Omega^{0q}(L)$ be the sheaf of $L$ valued smooth $(0,q)$ forms on $X$.
Consider the complex 
\[\ov{\partial}:\Gamma(X,\Omega^{00}(L))\ra \Gamma(X,\Omega^{01}(L)),\] one have 
\[H^0(X,L)=\ker(\ov{\partial})\] and \[H^1(X,L)=\coker(\ov{\partial}).\]
There are natural pre-Hilbert space structure on $\Gamma(X,\Omega^{00}(L))$ and $\Gamma(X,\Omega^{01}(L))$ defined as follows: 
For $v_1,v_2\in \Gamma(X,\Omega^{00}(L))$, locally, let $\alpha$ be a norm $1$ section of $\Omega^{00}(\Omega_{X})$ with respect to $g$,
\[\pair{\ ,\ }_{L^2}\] on $\Gamma(X,\Omega^{00}(L))$ is the integral over $X$ of 
\[\frac{-1}{2\pi i}\pair{v_1,v_2}\alpha\wedge\ov{\alpha},\]where $\pair{\ ,\ }$ is the Hermitian structure on $L$.
Let $u_1,u_2\in \Gamma(X,\Omega^{01}(L))$, locally, $u_i$ is of the form $\beta_iv_i$, 
\[\pair{\ ,\ }_{L^2}\]  on $\Gamma(X,\Omega^{01}(L))$ is the integral over $X$ of 
\[\frac{-1}{2\pi i}\pair{v_1,v_2}\ov{\beta_2}\wedge \beta_1.\]
Thus, there is a natural Hermitian structure on $H^0(X,L)$.
Hodge theory gives 
\[\Gamma(X,\Omega^{01}(L))=\ov{\partial}(\Gamma(X,\Omega^{00}(L)))\oplus \ker \ov{\partial}^*\] and hence there is a natural Hermitian structure on $H^1(X,L)$. 

The pairing 
\[\Gamma(X,\Omega^{00}(L))\times \Gamma(X,\Omega^{01}(\Omega_X\otimes L^\vee))\xlongrightarrow{\frac{1}{2\pi i}\int_X} \BC,\]
induces norm-preserving duality between the two complexes 
\[\Gamma(X,\Omega^{00}(L))\xrightarrow{\ov{\partial}_{L}}\Gamma(X,\Omega^{01}(L))\]
 and 
 \[\Gamma(X,\Omega^{00}(\Omega_X\otimes L^\vee))\xrightarrow{\ov{\partial}_{\Omega_X\otimes L^\vee}}\Gamma(X,\Omega^{01}(\Omega_X\otimes L^\vee))\] and hence an isometry on 
 \begin{equation}\label{sd}H^0(X,L)\simeq H^1(X,\Omega_X\otimes L^\vee)^\vee.\end{equation} 
 
One have metric $\pair{\ ,\ }_{L^2}$ on $\det R\Gamma(X,L)$ induced by the $L^2$-norm on $H^i(X,L)$ as above. 

\subsubsection{Quillen metric and Riemann-Roch at $\infty$}\label{Sqmrr}
 Let $\Delta_L=\ov{\partial}_{L}^\vee\ov{\partial}_{L}$ be the Laplacian with spectrum $\sigma(\Delta_L)$. It's zeta function is defined by \[\zeta_L(s)=\sum_{\lambda\in\sigma(\Delta_L)}\frac{1}{\lambda^s}\] which has analytic continuation. Recall that the regularized determinant(analytic torsion) $\det \Delta_{L}$ of $\Delta_L$ is defined by $e^{-\zeta_{L}'(0)}$. Then one have \begin{equation}\label{lap}\det \Delta_L=\det \Delta_{\Omega_X\otimes L^\vee}.\end{equation}
 The Quillen metric on $\det R\Gamma(X,L)$ is defined by \begin{equation}\label{l2q}\pair{\ ,\ }_{Q}=(\det \Delta_L)^{-1}\pair{\ ,\ }_{L^2}.\end{equation}
As $L$ varies, Quillen metric varies smoothly (cf.~\cite{Quillen:85}).
The Quillen metric also has the following properties: Let $\pair{\ ,\ }$ be the Deligne pairing on Hermitian line bundles of $X$, then the main result of \cite{Deligne:85} establishs the following isometry:
 \[(\det\nolimits_Q R\Gamma(X,L))^{\otimes 2}=\pair{L,L\otimes\Omega_X^\vee}\otimes (\det\nolimits_Q R(\Gamma,\CO_{X}))^{\otimes 2}.\]
  We may use $\det_{L^2} R\Gamma(X,L)$ and $\det_{Q} R\Gamma(X,L)$ to distinct two metrics.
\subsubsection{Riemann-Roch for arithmetic surfaces}\label{arr}
Let $f:\CX\ra \Spec\CO$ be an arithmetic surface. Choose the Arakelov metric on $\omega_{\CX/\CO}$. For $\CL$ be a Hermitian line bundle on $\CX$, take metric on $\det f_*\CL$ to be either Faltings metric or Quillen metric. Both of them depend on the Arakelov metric.
The following Riemann-Roch formula is due to Faltings \cite{Faltings}, Deligne \cite{Deligne:85} and Gillet-Soul\'{e} \cite{Gillet-Soule}:
Let  $\wh{\deg}$ is the arithmetic degree on $\wh{\Pic}(\CO)$ and $(\ ,\ )$ is the Arakelov arithmetic intersection pairing on $\wh{\Pic}(\CX)$, then
\begin{equation}\label{arrf}\wh{\deg}\det Rf_* \CL=\frac{1}{2}(\CL,\CL\otimes\omega_{\CX/\CO}^\vee)+\wh{\deg}\det Rf_* \CO_{\CX}.\end{equation}
We will use $\det_{F} Rf_* \CL$ and $\det_{Q} Rf_* \CL$ to distinct these two metrics.
Then they are connected by the following \cite[Sec.~4.5]{Soule:94}
\[\wh{\deg}(\det\nolimits_{Q} Rf_*\CL)^2:=\wh{\deg}(\det\nolimits_F Rf_*\CL)^2\prod_{v|\infty}e^{\zeta_{\CX_v(\BC)}'(0)}\pi^{-g}\vol(X_v(\BC))\]where $\zeta_{\CX_v(\BC)}(s)=\zeta_{\CO_{\CX_v(\BC)}}(s)$ and the volume is with respect to Arakelov metric.

\subsection{Numerical cohomology}
We recall basic properties of numerical cohomology for arithmetic curves and the related upper bound in the geometry of numbers in \S\ref{efco}. Then, in \S\ref{effsur} we introduce numerical cohomology for arithmetic surfaces and the relation to the Quillen metric on determinant bundle.
\subsubsection{Numerical cohomology for Arithmetic curve}\label{efco}
Let $\CO$ be the ring of integers of a number field $F$. Recall that a Hermitian coherent module $\CF$ on $\Spec\CO$ is a coherent $\CO$-module $\CF$ together with a Hermitian metric on $\CF_v\otimes_{F_v} \BC$ for each infinite place $v$. Let $||\cdot||$ be the induced metric on $\CF\otimes_{\BZ} \BR$. It has the property that \[-\log\ \frac{\covol(\CF)}{\#\CF_{t}}=\wh{\deg}\det\CF-\rank_{\CO}\CF\cdot \log\sqrt{|D_F|}.\]

For each infinite place $v$ of $F$, there is a natural way to equip the usual coherent module with a "trivial" Hermitian structure with metric on $\CF_v\otimes_{F_v}\BC$ given by \[||1||_v^2=\begin{cases}1,&\quad \text{$v$ is real},\\ 
2,&\quad \text{$v$ is complex}.\end{cases}\]
The sheaf of relative differential $\omega_\CO:=\omega_{\CO/\BZ}$ equipped with the "trivial" Hermitian structure plays an analogue role as the canonical sheaf in the geometric case for an algebraic curve over a finite field. 
Recall that \[h_{\CO}^0(\CF):=\log\sum_{x\in \CF}e^{-\pi ||x||^2},\quad h^1_{\CO}(\CF):=h_{\CO}^0(\omega_{\CO}\otimes \CF^\vee)\] and \begin{equation}\label{rrc1}\chi_\CO(\CF):=h_\CO^0(\CF)-h_\CO^1(\CF).\end{equation}
Then the Riemann-Roch formula for $\CF$ says that 
\begin{equation}\label{rrc}\chi_\CO(\CF)=-\log\frac{\covol(\CF)}{\#\CF_{t}}=\wh{\deg}\det\CF+\rank_{\CO}\CF \cdot \chi_\CO(\CO)
\end{equation}
with $\chi_\CO(\CO)=-\log\sqrt{|D_F|}$ (cf.~\cite{Geer-Schoof}).

We end this section by recalling the following results on upper bounds for $h_{\CO}^0$ and $\chi_{\CO}$.
\begin{prop}\label{bd}Let $\CF$ be a Hermitian vector bundle over $\Spec \CO$, then the following holds. 
  \begin{itemize}
   \item [(i)]\cite[Prop.~5.4]{Groenewegen}Let $\lambda$ be the minimum of $\CF$ and let $\gamma_i<i$ be the $i$-th Hermitian constant for $1\leq i\leq n:=\rank_{\BZ}\CF$, then \[h_\CO^0(\CF)\leq nh^0_{\BZ}(\BZ)+\sum_{i=1}^{n}\log\max\left\{1,\frac{\gamma_i}{\lambda}\right\}.\]
   \item [(ii)]\cite[Prop.~2.7.1]{Bost:2020}If $\CF$ is a Hermitian line bundle with $\wh{\deg}\CF\leq 0$, then 
   \[h^0_\CO(\CF)\leq 3^{[K:\BQ]}(1-\frac{1}{2\pi})e^{-\pi[F:\BQ]e^{-\frac{2}{[F:\BQ]}\wh{\deg}\CF}}\leq 1.\]
    \item [(iii)]\cite[Prop.~2.7.2]{Bost:2020} If $\CF$ is a Hermitian line bundle with $\wh{\deg}\CF\geq 0$, then 
   \[h^0_\CO(\CF)\leq 1+\wh{\deg}\CF.\]
  \end{itemize}
\end{prop} 
\begin{remark}
Part (i) follows from Minkowski theorem. Parts (ii) and (iii) are parallel to the geometric case.
\end{remark}
\begin{prop}[Minkowski's second theorem]\label{Min2}
Let $\CF$ be a Hermitian vector bundle over $\Spec \CO$ and let $n=\rank_{\BZ}\CF$. For $1\leq i\leq r:=\rank_{\CO}\CF$, let \[\lambda_i:=\inf\{\lambda\ |\ \text{exists $i$ $F$-linear independent vectors in $\CF$ with $||\cdot||\leq \lambda$}\}\] be the successive minima of $\CF$, then 
\[\chi_{\CO}(\CF)\leq  [F:\BQ]\sum_{i=1}^{r}(-\log\lambda_i)+c,\]where \[c=O(n\log r)\] given by $r[F:\BQ]\log 2-\log (V_{r}^{r_1}(2V_{2r})^{r_2})$ with $V_k$ the usual Lebesgue volume of $k$-dimensional unit ball in $\BR^k$. 
\end{prop}
\begin{remark}\ \label{ns}
\begin{itemize}
  \item [(i)]The Proposition \ref{Min2} has the following geometric analogue: Let $C$ be a projective smooth curve over a field and let $\CF$ be a rank $r$ semi-stable Hermitian vector bundle on $C$. Let $0=\CF_0\subset \CF_1\subset \cdots\subset \CF_k=\CF$ be the Harder-Narasimhan filtration of $\CF$. 
One have 
\[\deg \CF=\sum_{i=1}^k r_i\cdot \mu_i,\]where $r_i=\rank \CF_{i}/\CF_{i-1}$ and 
\[\mu_i=\frac{\deg\CF_{i}/\CF_{i-1}}{\rank \CF_{i}/\CF_{i-1}}\] is the slope of $\CF_{i}/\CF_{i-1}$.
\item [(ii)] For Hermitian $\CO$ vector bundle $\CF$, an analogue of Harder-Narasimhan filtration was introduce in \cite{Stuhler} and \cite{Grayson}.

For $1\leq i\leq r=\rank_{\CO}\CF$, let $\wh{\mu}_i$ be the $i$-th slope of the filtration introduce by Bost in \cite{Bost:1996}. Then $\wh{\mu}_i$ and $-\log\lambda_i$ in Proposition~\ref{Min2} are closely related \cite{Soule:97,Borek}.
\end{itemize}

\end{remark}
\subsubsection{Numerical cohomology for Arithmetic surface}\label{effsur}
Let $f:\CX\ra \Spec\CO$ be an arithmetic and let $\CL$ be a Hermitian line bundle on $\CX$. As in the introduction, we use Serre spectral sequence, Serre-dual and numerical cohomology of Hermitian coherent $\CO$-modules to define numerical cohomology of $\CL$. 
Let $\omega_{\CX/\CO}$ be the relative dualizing sheaf and let $\omega_{\CX}=\omega_{\CX/\BZ}$ which is the arithmetic analogue of canonical sheaf of algebraic surface in the geometric case\footnote{For example, it satisfies the following adjunction formula that for $C\simeq \Spec \CO$ an horizontal curve in $\CX$, then $\wh{\deg}\omega_{C}=(C,C+\omega_{\CX})$ holds.}.
\begin{defn}\
  \begin{itemize}
    \item [(i)]Define the numerical cohomology of a Hermitian line bundle $\CL$ on $\CX$ by 
\[h^0_{\CX}(\CL):=h^0_{\CO}(f_*\CL),\] \[\begin{aligned}h_\CX^1(\CL):=&h_\CO^0(\omega_\CO\otimes (f_*\CL)^\vee)+h_\CO^0(f_*(\omega_{\CX/\CO}\otimes\CL^\vee)^\vee)\\
  +&\frac{1}{2}(\wh{\deg}\det H^1(\CX,\CL)_t+\wh{\deg}\det H^1(\CX,\omega_\CX\otimes\CL^{\vee})_t),\\
\end{aligned}\] and \[h^2_{\CX}(\CL):=h^0_{\CX}(\omega_{\CX}\otimes \CL^\vee).\]
\item [(ii)]The Euler characteristic of the numerical cohomology is defined by 
\[\chi_{\CX}(\CL)=h^0_\CX(\CL)-h^1_{\CX}(\CL)+h^2_\CX(\CL).\]
  \end{itemize}

\end{defn}
\begin{remark}
One have \[\begin{aligned}
  h^1(\CL)=&h_\CO^0(\omega_\CO\otimes (f_*\CL)^\vee)+h_\CO^0(\omega_\CO\otimes(f_*(\omega_\CX\otimes\CL^\vee))^\vee)\\
  +&\frac{1}{2}\left(\wh{\deg}\det H^1(\CX,\CL)_t+\wh{\deg}\det H^1(\CX,\omega_\CX\otimes\CL^{\vee})_t\right)\\
=&h^1(\omega_\CX\otimes\CL^\vee)
\end{aligned}\]
and \begin{equation}\label{eu}\begin{aligned}
  \chi_{\CX}(\CL)=&\chi_\CO(f_*\CL)+\chi_\CO(f_*(\omega_\CX\otimes\CL^\vee))\\
  -&\frac{1}{2}\left(\wh{\deg}\det H^1(\CX,\CL)_t+\wh{\deg}\det H^1(\CX,\omega_\CX\otimes\CL^{\vee})_t\right)\end{aligned}
\end{equation}
\end{remark}

Let $\det_{L^2}Rf_*(\CL)$ be the invertible $\CO$-module $\det Rf_*(\CL)$ equipped with $L^2$-norm in Section \s\ref{pre} with $\omega_{\CX}$.
One has the following relation between $\chi_{\CX}(\CL)$ and $\det_{L^2}Rf_*(\CL)$.

\begin{prop}\label{com}
The following equality holds: 
\[
 \chi_{\CX}(\CL)=\wh{\deg}\det\nolimits_{L^2}Rf_*\CL-\frac{1}{2}(\CL,f^*\omega_{\CO})+\frac{1}{4}(\omega_{\CX},f^*\omega_{\CO}).
\]

\end{prop}
\begin{proof}
Note that
  \[\begin{aligned}
    \wh{\deg}\det\nolimits_{L^2} Rf_* \CL=\wh{\deg}\det\nolimits_{L^2}H^0(\CX,\CL)-\wh{\deg}\det\nolimits_{L^2}H^1(\CX,\CL)_f-\wh{\deg}\det H^1(\CX,\CL)_t.
  \end{aligned}\]

Recall relative Serre duality gives an isomorphism of Hermitian coherent $\CO$-modules:
  \[H^1(\CX,\CL)^\vee\simeq H^0(\CX,\omega_{\CX/\CO}\otimes \CL^\vee).\]
  Also note that for coherent Hermitian $\CO$-module $V$ and  a fractional ideal $W$ of $F$ with standard metric at $\infty$, \[\det(V\otimes W)=W^{\otimes\rank_{\CO}V}\otimes \det V.\] Together with the fact that \[(f^*\omega_\CO)^{\vee}\otimes \omega_\CX=\omega_{\CX/\CO},\] one have 
  \[\begin{aligned}
    \wh{\deg}\det\nolimits_{L^2} Rf_* \CL=&\wh{\deg}\det\nolimits_{L^2} H^0(\CX,\CL)+\wh{\deg}\det\nolimits_{L^2} H^0(\CX,\omega_{\CX}\otimes \CL^{\vee})\\
    -&\wh{\deg}\det H^1(\CX,\CL)_t-\rank_{\CO}H^1(\CX,\CL)\cdot\wh{\deg}\omega_\CO.
  \end{aligned}\]
  By Riemann-Roch for Hermitian coherent modules on $\Spec \CO$ (cf.~\eqref{rrc}), we have 
  \[\begin{aligned}
  \wh{\deg}\det\nolimits_{L^2} Rf_* \CL=&\chi_\CO(H^0(\CX,\CL))-\wh{\deg}\det H^1(\CX,\CL)_t+\chi_\CO( H^0(\CX,\omega_{\CX}\otimes \CL^\vee))\\
  -&\rank_{\CO}H^0(\CX,\CL)\cdot\chi(\CO)-\rank_{\CO}H^0(\CX,\omega_{\CX}\otimes \CL^\vee)\cdot\chi(\CO)\\
  -&\rank_{\CO}H^1(\CX,\CL)\cdot \wh{\deg}\omega_\CO
  \end{aligned}\]
  Now it follows from Riemann-Roch for Riemann surfaces, the fact that $\wh{\deg}\omega_\CO=-2\chi(\CO)$ and definition of $\chi_{\CX}(\CL)$ (cf.~\eqref{eu}), we have
  \begin{equation}\label{eq1}\begin{aligned}
    \wh{\deg}\det\nolimits_{L^2} Rf_* \CL=&\chi_{\CX}(\CL)-\chi_{\CX_\infty}(\CL_\infty)\chi_{\CO}(\CO)\\
    +&\frac{1}{2}\left(\wh{\deg}\det H^1(\CX,\omega_{\CX}\otimes \CL^\vee)_t-\wh{\deg}\det H^1(\CX,\CL)_t\right).
  \end{aligned}\end{equation} 
It follows from arithmetic Riemann-Roch (cf.~\eqref{arrf}) and \eqref{lap} that
  \begin{equation}\label{eq2}
    \begin{aligned}
     &\wh{\deg}\det\nolimits_{L^2}Rf_*\CL-\frac{1}{2}(\CL,f^*\omega_{\CO})\\
    =&\wh{\deg}\det\nolimits_{L^2}f_*(\omega_{\CX}\otimes\CL^\vee)-\frac{1}{2}(\omega_{\CX}\otimes\CL^\vee,f^*\omega_{\CO}). 
    \end{aligned}\end{equation}
  It follows from equations \eqref{eq1} and \eqref{eq2} that 
  \[\begin{aligned}
    &-\chi_{\CX_\infty}(\CL_\infty)\chi_{\CO}(\CO)+\frac{1}{2}\left(\wh{\deg}\det H^1(\CX,\omega_{\CX}\otimes \CL^\vee)_t-\wh{\deg}\det H^1(\CX,\CL)_t\right)\\
    =&\frac{1}{2}(\CL,f^*\omega_{\CO})-\frac{1}{4}(\omega_{\CX},f^*\omega_\CO)
  \end{aligned}\] and hence 
\[\begin{aligned}
 \chi_{\CX}(\CL)=&\wh{\deg}\det\nolimits_{L^2}Rf_*\CL-\frac{1}{2}(\CL,f^*\omega_{\CO})+\frac{1}{4}(\omega_{\CX},f^*\omega_{\CO})
\end{aligned}
\]holds.
\end{proof}

\subsection{Arithmetic Riemann-Roch formula for numerical  cohomology}\label{aarr}
In this section, we consider Riemann-Roch formula for numerical cohomology.

Let $\CX\ra \Spec\CO$ be an arithmetic surface and equip $\omega_{\CX}$ with Arakelov metric. For $\CL$ a Hermitian line bundle, let $h^i_\CX(\CL)$ be the numerical cohomology of $\CL$ and let \[\chi_{\CX}(\CL)=h^0_{\CX}(\CL)-h^1_{\CX}(\CL)+h^2_\CX(\CL)\] be the Euler characteristic defined in \S\ref{effsur}.
\begin{thm}\label{main:RR}Let $\CL$ be a Hermitian line bundle on arithmetic surface $\CX$, then
\[\left(\chi_\CX(\CL)+\frac{1}{2}\log \det \Delta_{\CL_\infty}\right)=\frac{1}{2}(\CL,\CL\otimes \omega_{\CX}^\vee)+\left(\chi_\CX(\CO_{\CX})+\frac{1}{2}\log \det\Delta_{\CO_{\CX_\infty}} \right)\]
\end{thm} 

\begin{proof} 
Recall that the Quillen metric and the $L^2$ metric are differed by regularized determinant of Laplacian (cf.~\eqref{l2q}), thus the following result follows from relative version of arithmetic Riemann-Roch (cf.~\eqref{arrf}) and Proposition~\ref{com}.
\end{proof}

\section{An upper bound for $(\omega_{\CX/\CO},\omega_{\CX/\CO})$}

In this section, let $f:\CX\ra \Spec\CO$ be an arithmetic surface. The main results are the upper bounds for $(\omega_{\CX/\CO},\omega_{\CX/\CO})$ (cf.~Theorem~\ref{main:2}~and~Theorem \ref{uppself1}).

\subsection{Noether formula for arithmetic surfaces}

In this section, we recall Noether formula for arithmetic surfaces, which is the starting point for analyzing the upper bound of $(\omega_{\CX/\CO},\omega_{\CX/\CO})$. 

The following Noether formula follows from the Noether formula of Faltings \cite{Faltings} and Moret-Bailly \cite{Moret-Bailly:1989}.

Recall that \[\chi_{\CX}(\CO_\CX)\] is the Euler characteristic for numerical cohomology   defined in \S\ref{effsur} and \[\det \Delta_{\CO_{\CX_\infty}}:=\prod_{v|\infty}\det \Delta_{\CO_{\CX_v(\BC)}}\]is the analytic torsion given in \S\ref{Sqmrr}.
\begin{prop}\label{Noe}
Let $f:\CX\ra \Spec\CO$ be an arithmetic surface, then 
\begin{equation}12\chi_{\CX}(\CO_{\CX})+\sum_{v|\infty}6{\epsilon_v}\log \det \Delta_{\CX_v(\BC)}=(\omega_{\CX},\omega_{\CX})+\sum_{v<\infty}\epsilon_v\delta_v+B(g)\end{equation}where $B(g)=O(g\cdot ([F:\BQ]+\log|D_F|))$ is an explicit constant and $\delta_v$ is the Artin conductor of $\CX_v$ for $v$ a finite place of $F$.
\end{prop}

\begin{remark}\label{Noermk}
  \
  \begin{itemize}
    \item [(i)]As in geometric case, $\chi_{\CX}(\CO_{\CX})$ is closely related to $\wh{\deg}f_*\omega_{\CX/\CO}$ via \eqref{eu} and the fact that \[\chi_{\CO}(f_*\omega_{\CX})-\chi_{\CO}(f_*\omega_{\CX/\CO})
  =(2g-2)\wh{\deg}\omega_{\CO}
 .\]
 \item [(ii)]Note that $|(\omega_{\CX/\CO},\omega_{\CX/\CO})-(\omega_{\CX},\omega_{\CX})|=O(g\log|D_F|)$, similar formula holds for $(\omega_{\CX/\CO},\omega_{\CX/\CO})$.
  \end{itemize}
\end{remark}
\begin{proof}
Recall that Noether formula states that (cf.~\cite[Sec.~4.6]{Soule})
\begin{equation}\label{noe}12\wh{\deg}\det\nolimits_{F}Rf_*\CO_{\CX}=(\omega_{\CX/\CO},\omega_{\CX/\CO})+\sum_{v}\epsilon_v\delta_v-\sum_{v|\infty}\epsilon_v4g\log 2\pi,\end{equation}where 
$v$ runs over all places of $F$, $\delta_v$ is the Artin conductor of $\CX_v$ for $v$ finite and $\delta_v$ is the Faltings $\delta$-invariant for $v|\infty$.

It is showed in \cite[Sec.~4.5]{Soule} that the Faltings metric and Quillen metric on $Rf_*\CL$ for a Hermitian line bundle $\CL$ on $\CX$ are connected by the following 
\begin{equation}
  \label{FaQ}\wh{\deg}\det\nolimits_{Q} Rf_*\CL=\wh{\deg}\det\nolimits_F Rf_*\CL-\sum_{v|\infty}\epsilon_v\log| \pi^{-g}(\det \Delta_{\CX_v(\BC)})^{-1}\vol(\CX_v(\BC))|^{1/2}\end{equation}where $\det \Delta_{\CX_v(\BC)}:=\det \Delta_{\CO_{\CX_v(\BC)}}$ and the volume is with respect to Arakelov metric. Furthermore, the Faltings invariant $\delta_v$ for $v|\infty$ and the analytic torsion are connected by the following (cf.~\cite[Sec.~4.6]{Soule},~\cite{Bost:1986})
\begin{equation}\label{daq}\delta_v=-6\log \frac{\det \Delta_{\CX_v(\BC)}}{\vol(\CX_v(\BC))} +A_v(g)\end{equation}where $A_v(g)=O(g)$ is an explicit constant only depends on $g$.
Combine equations \eqref{noe}, \eqref{FaQ}, \eqref{daq}, one has the following version of Noether formula in terms of Quillen metric 
\begin{equation}12\wh{\deg}\det\nolimits_{Q} Rf_*\CO_{\CX}=(\omega_{\CX/\CO},\omega_{\CX/\CO})+\sum_{v<\infty}\epsilon_v\delta_v+\sum_{v|\infty}\epsilon_v(A_v(g)-4g\log 2\pi+6g\log\pi)\end{equation}
Recall that by definition of Quillen metric, we have
\[\wh{\deg}\det\nolimits_{Q} Rf_*\CO_{\CX}=\wh{\deg}\det\nolimits_{L^2} Rf_*\CO_{\CX}+\sum_{v|\infty}\frac{\epsilon_v}{2}\log\det \Delta_{\CX_v(\BC)}\]
thus the result follows from Proposition~\ref{com} that we have
\begin{equation}12\chi_{\CX}(\CO_{\CX})+\sum_{v|\infty}6{\epsilon_v}\log \det \Delta_{\CX_v(\BC)}=(\omega_{\CX},\omega_{\CX})+\sum_{v<\infty}\epsilon_v\delta_v+B(g)\end{equation}where $B(g)=O(g\cdot ([F:\BQ]+\log|D_F|))$ given by $(f^*\omega_{\CO},\omega_{\CX})+\sum_{v|\infty}\epsilon_v(A_v(g)-4g\log 2\pi+6g\log\pi)$.
\end{proof}

\subsection{Upper bound for $(\omega_{\CX/\CO},\omega_{\CX/\CO})$ I}\label{uppI}

In this section, we obtain an upper bound for $(\omega_{\CX/\CO},\omega_{\CX/\CO})$ via the successive minima of $f_*\omega_{\CX/\CO}$ with respect to $L^2$-norm (cf.~Theorem~\ref{main:2}). As mentioned in the introduction and Remark~\ref{ns}, successive minima serve as the analogue of the slopes in the Harder-Narasimhan filtration in the geometric case.

Let briefly recall the $L^2$ norm on $f_*\omega_{\CX/\CO}$.

For $v|\infty$ an infinite place of $F$,
locally, let $a_v(z)$ be such that \[(f_1(z)dz,f_2(z)dz)\mapsto f_1(z)\ov{f_2(z)}a_v(z)\] gives Arakelov metric on $\omega_{\CX_v(\BC)}$. 
Then the Arakelov metric $\mu_{\Ar,v}$ on $\CX_v(\BC)$ given by \begin{equation}\label{arame}\frac{-1}{2 i}\frac{1}{a_v(z)}dz\wedge \ov{dz}.\end{equation} with $1$ has norm $\frac{1}{\sqrt{a_v(z)}}$ and the square of $L^2$-norm on $f_*\omega_{\CX_v(\BC)}$ is given by the integration of
\[\frac{-1}{2\pi i}f_1(z)\ov{f_2(z)}{dz\wedge \ov{dz}}.\]

In view of the Noether formula (cf.~Proposition~\ref{Noe}~and~Remark~\ref{Noermk}),
the upper bound for self-intersection number\[(\omega_{\CX/\CO},\omega_{\CX/\CO})\] reduces to upper bound for sum of Euler characteristic for numerical cohomology and analytic torsion.

Based on \eqref{eu}, to bound $\chi_{\CX}(\CO_{\CX})$ from above, it is enough to bound $\chi_{\CO}(f_*\CO_{\CX})$ and $\chi_{\CO}(f_*\omega_{\CX})$ from above.
Also recall that $\chi_{\CO}(f_*\omega_{\CX})$ and $\chi_{\CO}(f_*\omega_{\CX/\CO})$ are closely related via 
\begin{equation}\label{euabsre}\chi_{\CO}(f_*\omega_{\CX})-\chi_{\CO}(f_*\omega_{\CX/\CO})
  =(2g-2)\log |D_F|
 .\end{equation}

The following Lemma \ref{Min} is a direct consequence of Minkowski's second theorem (cf.~Proposition~\ref{Min2}).
\begin{lem}\label{Min}
Let $\lambda_i$, $i=1,\cdots, g$ be the successive minima of $f_*\omega_{\CX/\CO}$ with repsect to $L^2$-norm, then there exists an explicit constant $C=O(g\cdot([F:\BQ]\log g +|D_F|))$ such that 
 \[\chi_{\CO}(f_*\omega_{\CX/\CO})\leq [F:\BQ]\sum_{i=1}^{g}(-\log\lambda_i)+C.\]

\end{lem}
Note that $f_*\CO_{\CX}=\CO$, but the Hermitian structure is closely to volume of $\CX_\infty$ with respect to Arakelov metric. We have the following holds:
\begin{lem}\label{aravol}
There exists an explicit constant $C=O([F:\BQ]+|D_F|)$ such that 
\[\chi_{\CO}(f_*\CO_{\CX})\leq-\frac{1}{2}\sum_{v|\infty}\epsilon_v\log\vol(\CX_v(\BC))+C,\]where $\vol(\CX_v(\BC))$ is the volume of $\CX_v(\BC)$ with respect to Arakelov metric $\mu_{\Ar,v}$.
\end{lem}
\begin{proof}Recall that the Hermitian strcture on $f_*\CO_{\CX}=\CO$ is given by
  \[||1||_v^2=\epsilon_v\cdot \frac{1}{\pi }\int_{\CX_v(\BC)}\mu_{\Ar,v}.\]where $\mu_{\Ar,v}$ is the Arakelov metric on $\CX_{v(\BC)}$ as in \eqref{arame}. Thus the result follows from \eqref{rrc}.

\end{proof}
Combine with above ingredients and lower bound of Faltings $\delta$-invariants (cf.~\cite{Wilms:2017}), the following holds.
\begin{thm}\label{main:2}
Let $\lambda_i$, $i=1,\cdots, g$ be the successive minima of $f_*\omega_{\CX/\CO}$ with repsect to $L^2$-norm, then there exists an explicit constant $C=O(g\cdot([F:\BQ]\log g +|D_F|))$ such that 
\[(\omega_{\CX/\CO},\omega_{\CX/\CO})\leq 12[F:\BQ]\sum_{i=1}^{g}(-\log\lambda_i)+C.\]
\end{thm}
\begin{proof}

It follows from Noether formula (cf.~Proposition~\ref{Noe}~and~Remark~\ref{Noermk}) that
\[(\omega_{\CX/\CO},\omega_{\CX/\CO})\leq 12\chi_{\CX}(\CO_{\CX})+\sum_{v|\infty}6{\epsilon_v}\log \det \Delta_{\CX_v(\BC)}+O(g\cdot ([F:\BQ]+\log|D_F|))\] 

It follows from \eqref{eu}, \eqref{euabsre},Lemma \ref{Min} and Lemma \ref{aravol}, we have 
\[\begin{aligned}
  &12\chi_{\CX}(\CO_{\CX})+\sum_{v|\infty}6{\epsilon_v}\log \det \Delta_{\CX_v(\BC)}\\
  =&6\sum_{v|\infty}\epsilon_v \log \frac{\det \Delta_{\CX_v(\BC)}}{\vol(\CX_v(\BC))} - 12[F:\BQ]\sum_{i=1}^{g}\log\lambda_i+O(g\cdot ([F:\BQ]\log g+\log|D_F|))
\end{aligned}\]
Note that it follows from \eqref{daq} and lower bound of Faltings $\delta$-invariant (cf.~\cite[Cor.~1.2]{Wilms:2017})
\[\delta_v>-2g\log(2\pi^4),\] that 
\[\sum_{v|\infty}\epsilon_v \log \frac{\det \Delta_{\CX_v(\BC)}}{\vol(\CX_v(\BC))}=O(g[F:\BQ]),\]thus the result follows.
\end{proof}

\subsection{Lower bound for minimum of $f_*\omega_{\CX/\CO}$}\label{Upper1}
In this section, we assume the following 
\begin{itemize}
  \item [\tiny{$\bullet$}] $f:\CX\ra \Spec\CO$ admits a $\CO$-section $P$,
  \item [\tiny{$\bullet$}] $g\geq 2$.
\end{itemize}
The main result is an lower bound of minimum in $f_*\omega_{\CX/\CO}$ (cf.~Proposition~\ref{mainrpt}).

Let us first recall some invariants related to $P_\infty\in \CX_{\infty}$.

For each infinite place $v$ of $F$, we have a compact Riemann surface $\CX_v(\BC)$ of genus $g>1$. Then $\CX_v(\BC)\simeq \BH/\Gamma_v$, where $\BH$ is the upper half plane and $\Gamma_v\subset \PSL_2(\BR)$ is some torsion-free Fuchsian subgroup of the first kind that is isomorphic to the fundamental group of $\CX_v(\BC)$. 
Let $\pi_v: \BH\ra \CX_v(\BC)$ be the universal covering map.
 Let \[\mu_{\BH}(z)=\frac{i}{2}\frac{dz\wedge\ov{dz}}{\Im(z)^2}=\frac{dx\wedge dy}{y^2}\] be the hyperbolic measure on $\BH$ and $d_\BH$ is the hyperbolic distance. Given $z\in \BH$, recall that the (closure of) Dirichlet fundamental domain of $\CX_v(\BC)$ relevant to $z$ is given by the following convex compact subset
 \[D_{z,\Gamma_v}:=\{\tau\in \BH\ |\ d_{\BH}(z,\tau)\leq d_{\BH}(z,\gamma\tau),\ \forall \gamma\in \Gamma_v\}.\]  Let $r_z\in \BR_{>0}$ be maximal such that ${D_{z,\Gamma_v}}$ contains the ball $B(r_z,z)=\{\tau\ |\ |\tau-z|\leq r_z\}$.
 Choose a set of representatives $(z_v)_{v|\infty}\in (\BH)_{v|\infty}$ of $P_\infty$. 
 
 It turns out that $\log||\cdot||_{L^2}$ on $f_*\omega_{\CX/\CO}$ is bounded below by a constant related to $r_{z_v}$ and the multipler at $P$ defined as following.
\begin{defn}\label{mul}
Define the multipler $m_P$ of the local corrdinates of $P_\infty$ around $(z_v)_{v|\infty}\in (\BH)_{v|\infty}$ by 
\[m_P:=\prod_{v|\infty}|\pi_v^{*}(t_{P,v})'(z_v)|^{\epsilon_v},\]where $t_P$ is any element in $F(\CX)^\times$ that is a uniformizer of the formal completion of $\CX$ along $P$.
\end{defn}
By the product formula, $m_P$ is well defined.

Let \begin{equation}\label{loccor}c_P:=-\left(\min\log\{1,m_P^{2g-1}\}+\sum_{v|\infty}\frac{\epsilon_v}{2}\log \frac{2\min\{1,r_{z_v}^{4g-3}\}}{4g-3}\right).\end{equation}

We have the following lower bound of minimum of $f_*\omega_{\CX/\CO}$.
 \begin{prop}\label{mainrpt}
Let $(z_v)_{v|\infty}\in (\BH)_{v|\infty}$ be a set of representatives of $P_\infty\in (\CX_v(\BC))_{v|\infty}$, then for any $\omega\in f_*\omega_{\CX/\CO}$, we have
\[\log||\omega||_{L^2}\geq  -\log c_P.\]
\end{prop}
\begin{proof}
Let $\omega\in f_*\omega_{\CX/\CO}$ and let $r$ be the vanishing order of $\omega$ at $P$, then by Riemann-Roch, one have $0\leq r\leq 2g-2$. Let $t_P\in F(\CX)^\times$ be as in Definition \ref{mul}. We have that around $P$, \[\omega=\sum_{n\geq r}a_n\cdot t_P^n d t_P\] with $a_r\neq 0$ and $a_n\in \CO$ for all $n$.

Fix an infinite place $v$ of $F$.

View $F$ as subfield of $\BC$ via an embedding $\iota_v: F\ra \BC$ corresponding to $v$.
 Write $\varphi_v:=\pi_v^{*}t_P$ as \[\varphi_v(z)=\alpha_v\cdot(z-z_v)+\text{higher order terms}\] around $z_v$.

By definition of $||\cdot||_{L^2}$ on $f_*\omega_{\CX/\CO}$ (cf.~\S\ref{uppI}), we have
 \[\begin{aligned}
  \log||\omega||_{L^2}\geq& \sum_{v|\infty}\frac{\epsilon_v}{2}\log \left(\frac{-1}{2\pi i}\int_{B(r_{z_v},z_v)}\Big|\sum_{n\geq r}\iota_v(a_n)\cdot \varphi_v(z)^n\varphi_v'(z)\Big|^2dz\wedge \ov{dz}\right)\\
=&\sum_{v|\infty}\frac{\epsilon_v}{2}\log\left(\frac{1}{\pi}\int_{B(r_{z_{v}},z_v)}\Big|\iota_v(a_{r})\alpha_v^{r+1}\cdot(z-z_v)^r+\sum_{n\geq r+1}b_{n,v} (z-z_v)^n\Big|^2{dx\wedge dy}\right).
 \end{aligned}\]
 
Note that for $m\neq n$, 
\[\begin{aligned}
   &\int_{B(r_{z_v},z_v)}((z-z_v)^n(\ov{z-z_v})^m+(z-z_v)^m(\ov{z-z_v})^n)dx\wedge dy=0
\end{aligned}
\]
We have  
\[\log||\omega||_{L^2}\geq \sum_{v|\infty}\frac{\epsilon_v}{2}\log\left(\frac{1}{\pi}\int_{B(r_{z_{v}},z_v)}|\iota_v(a_{r})\alpha_v^{r+1}(z-z_v)^r|^2dx\wedge dy\right).\]  
Note that $a_r$ is an algebraic integer, thus we have 
\[\begin{aligned}
 \log||\omega||_{L^2}\geq&  \log m_P^{r+1}+\sum_{v|\infty}\frac{\epsilon_v}{2}\log \frac{2r_{z_v}^{2r+1}}{2r+1},
\end{aligned}
\]the result follows.
\end{proof}
\subsection{Upper bound for $(\omega_{\CX/\CO},\omega_{\CX/\CO})$ II}
Suppose that the  following holds:
\begin{itemize}
  \item [(A)] $\CX_F$ has genus $g>1$ and $\CX$ admits a $\CO$-section $P$.
\end{itemize} 
  Let $\lambda$ be the minimum of $f_*\omega_{\CX/\CO}$ and let \[C_P=12g[F:\BQ]\cdot \begin{cases}
 c_P,&\quad \text{if $\lambda<1$},\\
 0,&\quad \text{otherwise},
\end{cases}
\] where $c_P$ is given by \eqref{loccor}.

\begin{thm}\label{uppself1}
  There exists an explicit constant 
$C=O(g\cdot([F:\BQ]\log g +|D_F|))
$ such that 
\[(\omega_{\CX/\CO},\omega_{\CX/\CO})\leq C+C_{P}\]holds for $\CX$ satisfying assumption (A). 
\end{thm}
\begin{proof}
By Theorem~\ref{main:2}, it reduces to bound $L^2$-norms of global sections in $f_{*}\omega_{\CX/\CO}$ from below, thus the result follows from Proposition~\ref{mainrpt}.
\end{proof}
 
\end{document}